\documentclass[10pt]{amsart}

\usepackage[T1]{fontenc}
\usepackage[utf8]{inputenc}
\usepackage[a4paper,margin=30mm]{geometry}
\usepackage{lmodern}
\usepackage{microtype}
\usepackage{amsmath,amssymb,mathtools,amsthm,mathrsfs}
\usepackage{bm,booktabs,graphicx,float,tikz}
\usepackage[hidelinks]{hyperref}

\newtheorem{theorem}{Theorem}[section]
\newtheorem{proposition}[theorem]{Proposition}
\newtheorem{corollary}[theorem]{Corollary}

\newcommand{\R}{\mathbb{R}}

\newcommand{\dd}{\,\mathrm{d}}

\DeclareMathOperator{\aff}{aff}

\title{Direct Vertex-Derivative Coefficient Propagation for\\
Homogeneous Numerical Integration (HNI)}
\author{Jean B. Lasserre}
\thanks{This work has supported by the AI Interdisciplinary Institute ANITI funded by the France 2030 program under the Grant agreements 
ANR-19-P3IA-0004 and ANR-23-IACL-0002.}
\address{LAAS-CNRS, 7 avenue du Colonel Roche, BP 54200,\\
 31031 Toulouse, France}
\email{lasserre@laas.fr}
\date{}

\begin{document}

\maketitle
\begin{abstract}
Homogeneous numerical integration (HNI) combines Euler's identity and
Stokes' theorem to integrate homogeneous polynomials over a polytope $P$.
Existing algorithms return scalar integrals or reusable moment families.
Any resulting moment table can be
recoded afterward: once all degree-$q$ moments are known, integration on
$\mathcal H_q$ (degree-$q$ forms) has a trivial representation by a pure order-$q$ distribution
supported at any single arbitrary point.
Our contribution is instead to propagate functionals through the
face complex, once for $(P,q)$ and without first forming the moment vector,
to obtain a signed vertex-derivative coefficient functional on
$\mathcal H_q$, specific to $P$ and determined directly by its supplied face
geometry and chosen anchors.  A finite-dimensional Riesz framework describes its
nonuniqueness, gives an intrinsic $L^2(P)$ error for restricted formulas, and
supports weighted post-processing.  The independently generated degree-$q$
and degree-$2q$ coefficient vectors satisfy an explicit compatibility
relation.  Vertex anchors can prune the propagated representation as shown in some 
polygon examples and an offline benchmark quantify this effect without
claiming global cardinality optimality or improved conditioning.  The
construction covers convex polytopes and
oriented non-convex polyhedral domains in arbitrary dimension.
\end{abstract}

Riesz representative; cubature; moment problem.


\section{Introduction}
The basic setting is a compact full-dimensional convex polytope in the
affine space $\R^d$, written
\begin{equation}\label{eq:P-description}
 P=\{x\in\R^d:a_r^{\mathsf T}x\le b_r,
                 \ r=1,\ldots,s\},
 \qquad 0\in\operatorname{int}P,
\end{equation}
with an irredundant system of facet inequalities.  The interior-point choice
is for convenience 
but is not required by the
Euler--Stokes recurrence.

\paragraph{Geometric input.}
In the convex case, the implementation uses the vertex coordinates, the
face-incidence complex, and a chosen anchor in each positive-dimensional
face; these data can be generated from a standard $H$- or $V$-description of $P$.
The same notation $P$ is also used for the closure of a bounded, connected,
possibly non-convex polyhedral domain with nonempty interior.  In that case,
one must additionally supply its oriented boundary complex, including its
boundary components, incidence signs, and consistent outward side. 
It is \emph{not} a decomposition of $P$ into convex elementary cells.
Details are provided in the
Online Resources and the extended version; the mathematical extension is
specified in Section~2.

HNI combines
Euler's identity and Stokes' theorem and descends
from $P$ through its faces to the vertices.  The basic identities originate
in \cite{Lasserre83,Lasserre1998,Lasserre1999}; algorithms, non-convex extensions, and
applications on polytopic meshes appear in
\cite{MousaviSukumar2011,ChinLasserreSukumar2015,AntoniettiHoustonPennesi2018,RadleyHoustonHubbard2024}
for the two- and three-dimensional cases.
The last reference gives a graph organization, complexity analysis, and a
quadrature-free discontinuous Galerkin implementation.

HNI is also used as an assembly tool in several discretizations of PDEs.
Antonietti et al.~\cite{AntoniettiHoustonPennesi2018} treat high-order
discontinuous Galerkin methods on polytopic meshes; Benvenuti et
al.~\cite{BenvenutiEtAl2019} use HNI in an extended virtual element method for
singular and cracked Laplace problems; and Chin and Sukumar
\cite{ChinSukumar2019} and Chin et al.~\cite{ChinEtAl2021} apply it in
spectral extended finite elements on cut cells for elastostatics and
elastodynamics.  Bacuta et al.~\cite{BacutaEtAl2021} employ HNI for local
Gram and stiffness matrices in a three-dimensional DPG method, while Di
Pietro and Droniou~\cite{DiPietroDroniou2023} use it as an integration
subroutine in an arbitrary-order discrete de Rham method for magnetostatics.
Together with the quadrature-free transport implementation of
\cite{RadleyHoustonHubbard2024}, these applications motivate a short
clarification of the finite-dimensional functional that lies behind HNI as a
general integration tool on convex and non-convex polytopes in arbitrary
dimension.  This broader perspective is the main goal of the present
paper.

The present construction is intended as a complement to the efficient and
carefully developed algorithms of
\cite{ChinLasserreSukumar2015,AntoniettiHoustonPennesi2018,RadleyHoustonHubbard2024},
not as their replacement.  From a user's perspective, it may be viewed as a
functional ``X-ray'' of their internal Euler--Stokes machinery.  Those
algorithms return a scalar integral or reusable families of monomial and
intermediate face integrals, without repeating the complete descent for every
moment.  Their terminal vertex and derivative evaluations are already
implicit---and, for a supplied polynomial, operational.  The
proposed implementation retains the same face traversal but exposes and
collects contributions reaching each terminal functional $D^\beta|_v$ into a
polynomial-independent vector $C^{(q)}$.
If $E_q$ is its evaluation matrix on a monomial basis of $\mathcal H_q$, then
$b_q=E_qC^{(q)}$ is the resulting degree-$q$ moment vector.  
Thus the contribution is neither a new descent identity nor the first use of vertex derivatives. It is the explicit construction and retention of a reusable, polynomial-independent coefficient vector obtained by aggregating equal terminal vertex-derivative functionals during the face descent, rather than retaining only the final scalar integral or root moment vector returned by the cited algorithms. 
It makes anchor-dependent pruning explicit, supports restricted-information representations, and permits the degree-\(q\)/degree-\(2q\) compatibility relation to be written directly at coefficient level. No claim is made of lower cost or greater general reusability than optimized moment-table methods.


\noindent
{\bf Contribution.}
The principal contribution is conceptual as well as algorithmic.  Instead
of outputting only the canonical moment vector $b_q$, the proposed method
produces an explicit, geometry-adapted vertex-derivative representation
$C^{(q)}$ of the same integration functional.  An adapted
coefficient-propagation algorithm computes this structured lifting directly:
\begin{enumerate}
\item an adapted dynamic program explicitly produces the signed coefficient
vector representing integration on $\mathcal H_q$ directly from the oriented
face geometry, without first forming the moment vector or solving a global
representation system, and retains a reusable functional adapted to vertex
values and derivatives; the dynamic programming, reuse of lower-dimensional
face states, and anchor-based pruning of the existing HNI implementations
remain directly applicable. The output is changed from a
numerical monomial data to a sparse map indexed by terminal
vertex--derivative pairs;
\item the Riesz/least-squares identity describes nonuniqueness, provides
restricted formulas with an intrinsic $L^2(P)$ error, and supports weighted
post-processing;
\item coherent or exhaustively selected face-vertex anchors eliminate
recurrence branches and substantially prune the propagated output in the
reported convex and non-convex examples; and
\item the explicit outputs reveal a coefficient-level compatibility between
independently propagated HNI formulas at degrees $q$ and $2q$, and all
computational claims are accompanied by reproducible exactness checks.
\end{enumerate}
This is complementary to, rather than a replacement for, the efficient
moment-family implementations in the cited literature.  Online Resource~1
contains the complete formal propagation algorithm and the two- and
three-dimensional orientation requirements.  Online Resource~2 gives the
full inputs, coefficient output, and verification of the examples.  

\section{The reusable HNI functional}
Given an input degree $q\in\mathbb{N}$, set:
\[
 \mathcal H_q\,:=\{p\in\R[x_1,\ldots,x_d]:p(tx)=t^qp(x)\},\qquad
 m\,:=\dim\mathcal H_q=\binom{d+q-1}{q}.
\]
For the convex polytope in \eqref{eq:P-description}, a nonempty subset
$F\subseteq P$ is a face if there are $a\ne0$ and $\beta\in\R$ such that
\[
 P\subseteq\{x:a^{\mathsf T}x\le\beta\},\qquad
 F=P\cap\{x:a^{\mathsf T}x=\beta\};
\]
by convention, $P$ itself is also a face.  A $k$-face satisfies
$\dim(\aff F)=k$, where
\[
 \aff F=\left\{\sum_i\theta_ix_i:x_i\in F,\ \sum_i\theta_i=1\right\}.
\]
The $0$-faces are the vertices, denoted $\operatorname{vert}(P)$; the
$(d-1)$-faces are the facets.  If $F$ is a $k$-face, $G\prec F$ denotes one
of its relative $(k-1)$-facets, considered in $\aff F$, and $\mu_k$ denotes 
the $k$-dimensional Lebesgue measure on $\aff F$.

For the non-convex extension, $P$ is the closure of a bounded polyhedral
Lipschitz domain whose boundary is supplied as a compatible finite oriented
polyhedral complex.  It is \emph{not} decomposed into convex full-dimensional cells:
$P$ remains the single root domain, while its lower-dimensional states are
the boundary cells.  Thus $G\prec F$ means an oriented boundary facet and
terminal states include vertices introduced by an optional boundary
refinement. A simple two-dimensional illustrative example is displayed  in Figure 1.

\begin{figure}[ht]
\label{fig-1}
\centering
\includegraphics[width=.45\textwidth]{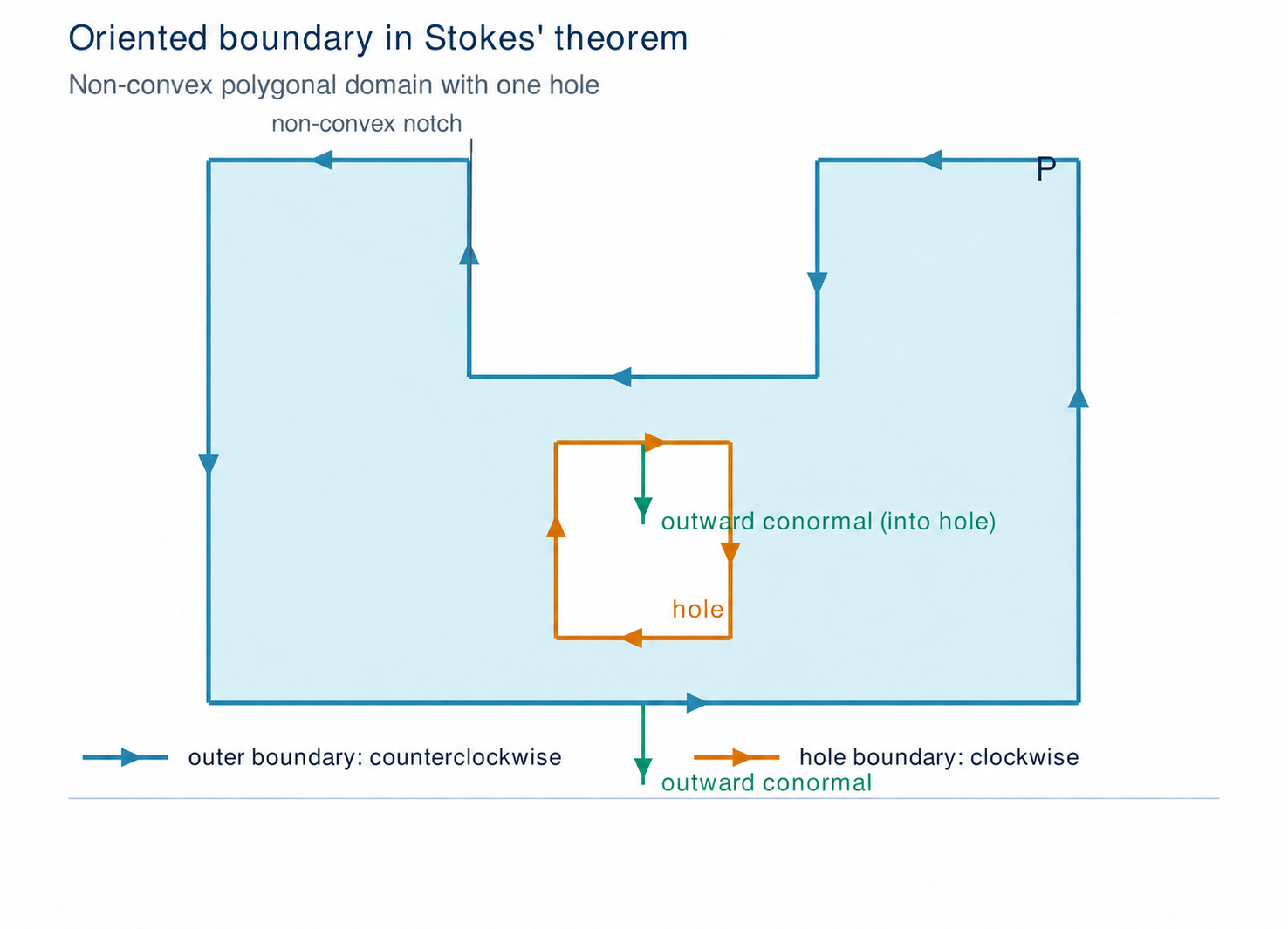}
\caption{{\small Non-convex polygonal domain with a hole: The exterior cycle is
counterclockwise, the hole cycle is clockwise, and every conormal points out
of the domain.  Thus inner-boundary conormals point into the hole.} }
\end{figure}

For an anchor $z_F\in\aff F$ and a relative facet $G\prec F$, let
$n_{F,G}$ be the outward unit conormal in $\aff F$ and
$h_{F,G}=(x-z_F)\cdot n_{F,G}$ for $x\in G$. 
Denote by \(\operatorname{Div}_F X\) the tangential divergence of a tangential vector field \(X\) in \(\operatorname{aff}F\), and by $\nabla_Fg\,(=\prod_F\nabla g)$ the projection of $\nabla g$ on $F$.

\begin{theorem}[Facewise Euler--Stokes recurrence]\label{th:rec}
For $p\in\mathcal H_q$, a $k$-face $F$, and $|\alpha|\le q$, put
$I(F,\alpha):=\int_FD^\alpha p\dd\mu_k$.  If $k\ge1$, then
\begin{equation}\label{eq:rec}
 (k+q-|\alpha|)I(F,\alpha)=
 \sum_{G\prec F}h_{F,G}I(G,\alpha)
 +\sum_{i=1}^d(z_F)_iI(F,\alpha+e_i),
\end{equation}
where the second sum is absent for $|\alpha|=q$.  At a vertex $v$,
$I(\{v\},\alpha)=D^\alpha p(v)$.
\end{theorem}
\begin{proof}
For $g=D^\alpha p$, apply the intrinsic Gauss--Green formula in $\aff F$ to 
the tangential vector field $X=(x-z_F)g$, and observe that
\[\langle \,x-z_F,\nabla_F g(x)\,\rangle\,=\,\langle \,x-z_F,\nabla g(x)\,\rangle\,,\quad\forall x\in\aff F\,.\]
Then  using by Euler's identity, $\mathrm{Div}_F((x-z_F)g)$ reads
$(k+q-|\alpha|)g-z_F\cdot\nabla g$. Its flux on $G$  is
$h_{F,G}\,g$, which yields \eqref{eq:rec}.  Finally, along every branch,
$k+q-|\alpha|$ decreases by one at each step,  and therefore one terminates at a vertex.
\end{proof}
Theorem \ref{th:rec} is a compact reformulation of the recurrence used in
\cite{Lasserre1998,ChinLasserreSukumar2015,AntoniettiHoustonPennesi2018,RadleyHoustonHubbard2024}.
The tangential divergence makes clear that Stokes' theorem is applied in the
successive affine spaces of lower-dimensional faces. 
Next, formal propagation of \eqref{eq:rec}, and collecting equal terminal
vertex--multi-index pairs $(\xi_j,\alpha^{(j)})$, yields
\begin{equation}\label{eq:hni}
 \int_Pp(x)\dd x=\sum_{j=1}^NC_j^{(q)}D^{\alpha^{(j)}}p(\xi_j),
 \qquad p\in\mathcal H_q.
\end{equation}
Fix an ordered monomial basis column $v_q = (x^\gamma)_{\vert\gamma\vert=q}\in\R^m$.
With $\lambda_j(p) := D^{\alpha(j)}p(\xi_j)$, define $E_q \in\R^{m\times N}$
by $(E_q)_{aj}:=\lambda_j((v_q)_a)$, and let $b_q: =\int_P v_q(x) dx\in \R^{m}$ . 
Thus \eqref{eq:hni} applied to the basis elements reads $b_q=E_q\,C^{(q)}$.
Equivalently, with the standard distributional convention,
\begin{equation}\label{eq:distribution}
 \mathcal T_{P,q}=
 \sum_{j=1}^N(-1)^{|\alpha^{(j)}|}C_j^{(q)}
 \partial^{\alpha^{(j)}}\delta_{\xi_j},\qquad
 \langle\mathcal T_{P,q},p\rangle=\int_Pp.
\end{equation}
The equality is only on $\mathcal H_q$; it is \emph{not} equality with Lebesgue
measure on an unrestricted test space.

The construction has two stages.  First, $(P,q)$, oriented incidence data,
and anchors determine the list $(\xi_j,\alpha^{(j)})$ and the vector
$C^{(q)}$, without selecting $p$.  Second, upon request, any polynomial $p\in\mathcal{H}_q$ is
integrated by the scalar sum \eqref{eq:hni}.  Thus the vector of moments 
$b_q$ is an output of HNI.

\paragraph{Coefficient-propagation algorithm.}
For each ``state" $(F,\alpha)$, store a sparse map $\mathscr{C}(F,\alpha)$ from terminal pairs
$(v,\beta)$ to coefficients.  Process states recursively by decreasing
$\dim F+q-|\alpha|$: initialize a vertex state by
$\{(v,\alpha)\mapsto1\}$; otherwise apply \eqref{eq:rec}, i.e. formally,
\[\mathscr{C}(F,\alpha)=\frac{1}{k+q-\vert\alpha\vert}\left[
 \sum_{G\prec F}h_{F,G}\mathscr{C}(G,\alpha)
 +\sum_{i=1}^d(z_F)_i\mathscr{C}(F,\alpha+e_i)\right],\]
memoize\footnote{store the result of a computation so that it can be reused when the same input or state occurs again.} each child
state, add its scaled map, collect duplicate terminal pairs, and remove exact
zeros.  The returned map at $(P,0)$ is $C^{(q)}$.  For floating-point geometry
the implementation uses a declared zero tolerance; oriented incidences enter
through the signed conormals in $h_{F,G}$.  This is the complete logical
pseudocode while Online Resource~1 provides data structures and detailed propagation.

If $b_\alpha=\int_Px^\alpha$ are known, then trivially
$\int_Pp=\sum_{|\alpha|=q}(b_\alpha/\alpha!)D^\alpha p(v)$ for any arbitrary
$v\in\R^d$ (even $v\not\in P$); no computation is involved.
This is the unrestricted $m_q=\dim\mathcal H_q$ term-count baseline: but it
pre-supposes the complete moment vector $b_q$ and uses every derivative of order
$q$ (which is just coefficient extraction).  The anchor counts below instead
measure pruning of the structured Euler--Stokes output.  The 12- and 31-term
post-processed formulae for the cube 
contain only
derivatives of order at most two; the counts for the polygon  are not claimed to beat
the unrestricted moment formula.

More generally, after any algorithm has produced the basis moment vector
$b_q$, a vertex-supported mixed-derivative vector can be sought by solving
$E_qC=b_q$ over any prescribed candidate family $C$.  Thus neither point support,
vertex support, nor existence of such a coefficient vector is the novelty.
The proposed algorithm constructs one such $C$ directly through local Euler--Stokes
propagation, retaining the face-incidence and anchor structure and avoiding
this moment-first global post-processing step.

For a general polynomial $f=\sum_{s=0}^Qf_s$, apply the degree-$s$ functional
to each homogeneous component and sum the results.  This decomposition is
relative to the chosen affine origin and changes under translation, although
the integral of $f$ does not.

By contrast, unknown-node positive moment-fitting cubature takes the moments
as input and searches for nodes in $P$ and nonnegative weights satisfying
nonlinear moment equations.  HNI fixes vertex support, avoids node location,
and instead uses signed derivative coefficients.  This comparison does not
apply to every cubature construction, in particular not to fixed-node linear
methods or compression of an already available positive formula.

Convexity is not used in the divergence identity.  For a non-convex polygon,
the ordered exterior cycle, oppositely oriented hole cycles, and their
induced outward conormals are required.  In three dimensions, consistently
oriented polygonal faces and signed face--edge--vertex incidences are needed.
Throughout, the incidence sign is absorbed into the parent-relative conormal
and hence into $h_{F,G}$; no second sign is inserted in \eqref{eq:rec}.  The
supplied boundary cells are assumed to form a consistently oriented finite
polyhedral complex.  No subdivision of the volume into simplices or convex
cells is inherent in the construction; Online Resource~1 gives the complete
input convention.

\section{Riesz representation and structured post-processing}
Equip $\mathcal H_q$ with $\langle f,g\rangle_P=\int_Pfg$.  For a basis
column $v_q(x)$ define
\[
 M_q\,:=\int_Pv_qv_q^{\mathsf T},\qquad b_q\,:=\int_Pv_q,
 \qquad K_{P,q}(x,y)\,:=v_q(x)^{\mathsf T}M_q^{-1}v_q(y),
\]
and
\begin{equation}\label{eq:Riesz}
 R_{P,q}(x)\,:=v_q(x)^{\mathsf T}M_q^{-1}b_q
           =\int_PK_{P,q}(x,y)\dd y.
\end{equation}
Then $\int_Pp=\langle p,R_{P,q}\rangle_P$.  This is the ``exact-degree" kernel
of $\mathcal H_q$, not the usual ``degree-at-most" Christoffel setting.  In
particular, $K_{P,q}(0,0)=0$ for $q>0$ because $q(0)=0$ for all
$q\in\mathcal H_q$. Next, for $\lambda_j(p)=D^{\alpha^{(j)}}p(\xi_j)$, set
\[
 \kappa_j(x)\,:=\left.D_y^{\alpha^{(j)}}K_{P,q}(x,y)
 \right|_{y=\xi_j},\qquad
 Q_C(p)\,:=\sum_jC_j\lambda_j(p).
\]
\begin{theorem}[Riesz residual and coefficient nonuniqueness]\label{th:compression}
For every $C\in\R^N$,
\begin{equation}\label{eq:wce}
 \sup_{\substack{p\in\mathcal H_q\\\|p\|_{L^2(P)}=1}}
 |\,\textstyle\int_Pp-Q_C(p)|
 =\left\|R_{P,q}-\sum_jC_j\kappa_j\right\|_{L^2(P)}.
\end{equation}
Let $(\phi_a)_{a=1}^m$ be an orthonormal basis, and 
$J_{aj}:=\lambda_j(\phi_a)$, $\rho_a:=\int_P\phi_a$.  Then:
\begin{enumerate}
\item $Q_C$ is exact iff $JC=\rho$, equivalently
$R_{P,q}=\sum_jC_j\kappa_j$;
\item all exact vectors in this information set form $C_0+\ker J$;
\item if $\rho\in\operatorname{range}J$ and $W\succ0$, the exact vector
minimizing $C^{\mathsf T}WC$ is
\[
 C_W=W^{-1}J^{\mathsf T}(JW^{-1}J^{\mathsf T})^\dagger\rho\,,
\]
where $(JW^{-1}J^{\mathsf T})^\dagger$ denotes the Moore-Penrose pseudoinverse;
\item after restricting the available derivatives, $J^\dagger\rho$ is the
minimum-norm least-squares rule and its squared worst-case error is
$\|\rho-JJ^\dagger\rho\|_2^2$.
\end{enumerate}
\end{theorem}
\begin{proof}
The reproducing property yields
$\int_Pp-Q_C(p)=\langle p,R_{P,q}-\sum_jC_j\kappa_j\rangle_P$.
Taking the functional norm proves \eqref{eq:wce}; orthonormal coordinates and
standard constrained least squares give the remaining statements.
\end{proof}

The Hilbert-space identity and pseudoinverse formulas are standard; their
role here is to certify and reduce the explicit HNI output.  A coefficient
norm is meaningful only after derivatives of different orders have been
scaled, so $W$ must encode a proposed normalization or evaluation cost.
In floating point, its is appropriate to use scaled or orthonormal polynomials and
QR or SVD (for rank).  The exact-rational reductions
below establish algebraic exactness, not a general stability result.

\begin{proposition}[Anchor changes]\label{prop:anchors}
Embed the exact outputs $C$ and $\widetilde C$ from two anchor choices in a
common vertex-derivative information set.  Then $J(C-\widetilde C)=0$.
Thus anchors can change the displayed coefficients, but not their induced
functional on $\mathcal H_q$.
\end{proposition}
The choice can change not only the coefficients but also their number.  If an
anchor lies on a relative facet, the corresponding height $h_{F,G}$ vanishes
and that boundary branch disappears.  Coherent face-vertex anchors therefore
can produce much sparser formal outputs, although in general they do not
solve the separate sparsest-representation problem.

\paragraph{A finite vertex-anchor search for polygons.}
The heuristic modeling choice is to restrict the root anchor of an oriented
$n$-gon to its vertices and every edge anchor to an endpoint.  Within this
class, testing all $n\,2^n$ assignments is exhaustive and returns a
global minimum of collected support; it does not optimize over arbitrary
affine anchors.  Its exponential cost precludes a general efficiency claim.
Nevertheless,
the cited PDE applications use HNI predominantly in dimensions two and
three.  For a fixed low-dimensional element with a moderate number of faces,
the analogous finite enumeration---restricting each face anchor to a vertex
of that face and each edge anchor to an endpoint---is an implementable offline
preprocessing step.  Its cost is then amortized because the selected
coefficient vector is reused in all subsequent polynomial integrations and
element assemblies.  Section~5 reports timings through $n=10$ and provides
reproducible non-convex and non-symmetric tests.

\paragraph{Compression does not require $M_q$.}
This point is to prevent a possible circular interpretation.  Let
$C^{\rm raw}\in\R^{N_{\rm raw}}$ be the coefficient vector produced by the
proved Euler--Stokes recurrence, let
$E_{\rm raw}\in\R^{m\times N_{\rm raw}}$ evaluate its functionals on the
chosen basis $v_q$, and let
$E_{\rm red}\in\R^{m\times N_{\rm red}}$ correspond to a proposed reduced
information set.  Exact algebraic compression simply reads
\begin{equation}\label{eq:compression-without-M}
 E_{\rm red}C^{\rm red}=E_{\rm raw}C^{\rm raw}.
\end{equation}
Thus neither $M_q$ nor a degree-$2q$ HNI calculation is needed to construct
or verify a compression relative to the already proved raw formula.  The
moment matrix is needed only to equip $\mathcal H_q$ with the intrinsic
$L^2(P)$ metric, form $R_{P,q}$ and the kernel, and quantify the worst-case
error of an inexact restricted formula.
Also as a useful and cheap sanity check, every exact formula must satisfy
\begin{equation}\label{eq:scalar}
 \langle\mathcal T_{P,q},R_{P,q}\rangle
 =\int_PR_{P,q}=\|R_{P,q}\|_{L^2(P)}^2=b_q^{\mathsf T}M_q^{-1}b_q\,,
\end{equation}
but full exactness is certified by \eqref{eq:wce}, not by this scalar identity
alone.

Interestingly, the explicit coefficients also carry a cross-degree compatibility that is
not apparent from the two independent Euler--Stokes propagations.

\begin{corollary}[Degree-$q$/degree-$2q$ link]\label{cor:cross-degree}
For $s\in\{q,2q\}$, let
\[
 m_s=\dim\mathcal H_s=\binom{d+s-1}{s},\qquad
 C^{(s)}\in\R^{N_s},\qquad
 E_s\in\R^{m_s\times N_s},
\]
where $v_s\in\R^{m_s}$ is a basis column and
$(E_s)_{aj}=\lambda_j^{(s)}((v_s)_a)$.  If the two HNI formulas are exact,
then
\begin{equation}\label{eq:cross-moments}
 b_q=E_qC^{(q)},\qquad
 (M_q)_{ab}=\sum_{j=1}^{N_{2q}}C_j^{(2q)}
 \lambda_j^{(2q)}((v_q)_a(v_q)_b).
\end{equation}
Consequently $r=M_q^{-1}b_q$ recovers
$R_{P,q}=v_q^{\mathsf T}r$.  Define
$\mathcal P(r)\in\R^{m_q\times m_{2q}}$ by
\begin{equation}\label{eq:product-matrix}
 v_q(x)\,R_{P,q}(x)\,=\,\mathcal P(r)v_{2q}(x).
\end{equation}
Then the coefficient vectors satisfy
\begin{equation}\label{eq:cross-link}
 \boxed{\ \mathcal P(r)\,E_{2q}\,C^{(2q)}=E_q\,C^{(q)}.\ }
\end{equation}
For fixed and known $r$, this is a linear relation between the two HNI
vectors, with block matrix
\[
 [\,\mathcal P(r)\,E_{2q}\; -E_q\,]
 \in\R^{m_q\times(N_{2q}+N_q)}.
\]
Conversely, exact HNI vectors at degrees $q$ and $2q$ give $b_q$ and $M_q$
through \eqref{eq:cross-moments}, and hence determine $r$.
\end{corollary}

\begin{proof}
The first identities are the two HNI formulas applied to the entries of
$v_q$ and to their products.  For each entry $h$ of $v_q$, the Riesz identity
gives $\int_Ph=\int_PhR_{P,q}$.  Applying respectively the degree-$q$ and
degree-$2q$ formulas and using \eqref{eq:product-matrix} gives
\eqref{eq:cross-link}.
\end{proof}

The relation is elementary once $M_qr=b_q$ is introduced, but its
coefficient-level interpretation is not.  The degree-$q$ vector alone does
not supply $M_q$: its entries have degree
$2q$, so the second HNI run or another exact moment engine is essential.  If
$M_q=LL^{\mathsf T}$ is a Cholesky factorization, then $L^{-1}v_q$ is an
$L^2(P)$-orthonormal family.  The moment recovery and orthonormalization are
standard; the HNI-specific feature is that \eqref{eq:cross-link} is written
directly in terms of the two reusable vertex-distribution coefficient
vectors.  When both $b_q$ and $M_q$ are computed by the same HNI recurrence
at degrees $q$ and $2q$, the Riesz and cross-degree identities are useful
internal consistency checks and post-processing relations, but they are not
independent validations of HNI.  An external numerical validation must use
moments obtained analytically or from an independent integration
implementation, as in Section~5.

\section{Positive value-only representation}
Applied to a basis of $\mathcal H_q$, the overlooked Richter's theorem for measurable
integrable functions gives a positive value-only formula with nodes in $P$
and at most $\binom{d+q-1}{q}$ terms, sharper on this exact-degree space than
the Tchakaloff bound $\binom{d+q}{q}$ for $\mathbb P_{\le q}$
\cite{Richter1957,schmudgen,Putinar1997,BayerTeichmann2006}.  It is an
existence result and does not locate the nodes; HNI instead fixes vertex
support and uses signed derivatives.

\section{Examples and numerical checks}
For $B=[-1,1]^3$, symmetry and product moments give
\begin{equation}\label{eq:boxR}
 R_{B,4}(x)=\frac{595}{363}\sum_{i=1}^3x_i^4
 -\frac{350}{363}\sum_{i<j}x_i^2x_j^2.
\end{equation}
Using face centers as anchors, the formal recurrence produces 31 nonzero
multi-indices at each of eight vertices, hence 248 terms.  On the 15
monomials of $\mathcal H_4$, the resulting $15\times248$ evaluation matrix
has exact rank 15 and nullity 233.  Alternatively, encode a face by
$s\in\{-1,0,1\}^3$ and choose the coherent vertex anchor
$(z_s)_i=s_i$ when $s_i\ne0$ and $(z_s)_i=-1$ otherwise.  Every incident
facet through $z_s$ then has zero height, and the recurrence has only 22
nonzero terms, all supported at the opposite vertex $(1,1,1)$.  Exact
rational product moments verify all 15 identities.  Thus anchor selection
provides a drastic systematic reduction of the formal output, though it does
not by itself produce a smallest rule.  For a reproducible algebraic
post-processing, take as candidate information all derivatives of order at
most two at the four vertices of the facet $x_1=-1$.  Order the resulting 40
columns lexicographically by vertex, total derivative order, and multi-index.
The exact $15\times40$ rational evaluation matrix has rank 15.  Reduced row
elimination from left to right, with every nonpivot unknown set to zero,
yields an exact 12-term formula supported at three vertices of that single
face.  This basic solution is not canonical and changes if the columns are
reordered.  Online Resource~2 displays its coefficients.  Exact rational
arithmetic verifies all 15 monomial identities and the defining Riesz
identities for \eqref{eq:boxR}.

\noindent
{\bf Restricted-data interpretation.}
The practical distinction from a moment table concerns the representation in
which the integrand is supplied.  The vector $b_4$ is directly reusable when
the 15 monomial coefficients of $p\in\mathcal H_4(\mathbb R^3)$ are known.
The reduced formula is directly reusable when instead the indicated local
vertex data are available: it recovers the integral from 12 values and
derivatives of order at most two, without reconstructing all 15 monomial
coefficients or evaluating the complete fourth-order derivative family of
the one-point Taylor representation.  The reduced rule computes the scalar integral \(\int_Pp\) from these 12 local data only.
This is an information-model distinction, not a universal cost claim.

Next, consider $B_4=[-1,1]^4$ and choose face centers as anchors.  Put
\[
 \mathcal A_4^{(4)}=\{\alpha\in\mathbb N_0^4:|\alpha|\le4,
                         \ \#\operatorname{supp}\alpha\le3\}.
\]
The same path argument gives
\begin{equation}\label{eq:four-cube-HNI}
 \int_{B_4}p=\sum_{\varepsilon\in\{\pm1\}^4}
 \sum_{\alpha\in\mathcal A_4^{(4)}}
 \bar c_\alpha\varepsilon^\alpha D^\alpha p(\varepsilon),
 \qquad p\in\mathcal H_4,
\end{equation}
where $\bar c_\alpha$ depends only on the decreasing rearrangement of
$\alpha$; it is constant on each coordinate-permutation class of multi-indices. The 11 resulting distinct coefficients are tabulated in Online Resource 2.
There are 69 multi-indices per vertex and hence $16\times69=1104$ nonzero
face-center terms.  With the analogous coherent face-vertex rule, the formal
recurrence instead has 35 nonzero terms, all supported at $(1,1,1,1)$, and
has exact residual zero on the 35 monomials of $\mathcal H_4$.  On derivatives
of order at most two at the eight vertices of $x_1=-1$, the $35\times120$
evaluation matrix has rank 35 and exact elimination gives a 31-term formula.
Thus $1104\to31$, without a claim that 31 is minimal.  In general, the rank
bound $\dim\mathcal H_4(\mathbb R^4)=35$ ensures a basic solution with at most
35 terms in any spanning candidate family, but finding a sparsest support is
left as a separate combinatorial problem.

Direct product moments provide an independent exact reference without SymPy.
Formula \eqref{eq:four-cube-HNI}, the 35-term vertex-anchor output, and the
31-term compression agree exactly
with that reference on all 35 monomials of $\mathcal H_4$ and on 100 random
quartics with integer coefficient.  Its Riesz identity also has exact residual
zero.  Online Resource~2 gives the Riesz polynomial, all 11 classes,
and CSV files for the 1104-, 35-, and 31-term formulas.

Table~\ref{tab:checks} summarizes independent checks.  Every discrepancy is
an absolute error: the first number is the maximum on the complete monomial
basis and the second the maximum on 100 seeded random polynomials.  
For the rational examples, a zero reported for the basis residual is an exact zero obtained in rational arithmetic. Any nonzero residual reported for the random-polynomial test arises only from the subsequent binary64 evaluation.
The hexagon rows
use binary64 geometry, collect coefficients below $10^{-12}$ as zero, and
report both observed floating-point maxima.  The proposed
coefficient vector was formed once for each $(P,q)$ and reused on the full
monomial basis and 100 random polynomials.  Reference moments came from
product integration or oriented triangle integration and, independently,
from SymPy's \texttt{polytope\_integrate}, whose documentation points to
\cite{ChinLasserreSukumar2015}.  SymPy \cite{MeurerEtAl2017} was used as a moment engine and does
not itself output our explicit vertex-distribution coefficients.
All term counts depend on the anchor policy, which is stated in the table.
These counts measure the size of the propagated representation relative to
the unrestricted $m_q$ baseline displayed in the table, not wall-clock cost
or numerical stability: derivative order and normalization matter.
\begin{table}[t]
\caption{Propagated distributions, unrestricted moment baseline, and
verification (basis maximum/random-sample maximum).  ``Best vertex'' means
best only within the declared finite anchor class.}
\label{tab:checks}\centering\small
\resizebox{\textwidth}{!}{\begin{tabular}{lrrrc}
\toprule
domain, degree & propagated/restricted terms & $m_q$ & reference & max discrepancy\\
\midrule
cube, $q=4$ &248 center; 22 vertex; 12 order-$\le2$&15& exact product moments &$0/1.78\!\times\!10^{-15}$\\
four-cube, $q=4$ &1104 center; 35 vertex; 31 order-$\le2$&35& exact product moments &$0/0$\\
non-convex L-polygon, $q=4$ &54 sparse affine; 30 best vertex&5& exact triangle/SymPy &$0/0$\\
seeded non-symmetric convex hexagon, $q=4$ &90 midpoint; 30 best vertex&5& oriented triangle &$5.56\!\times\!10^{-17}/1.11\!\times\!10^{-16}$\\
seeded non-symmetric convex hexagon, $q=8$ &270 midpoint; 90 best vertex&9& oriented triangle &$1.39\!\times\!10^{-17}/4.17\!\times\!10^{-17}$\\
\bottomrule
\end{tabular}}
\end{table}

\noindent
{\bf Cost of the finite anchor search.}
A separate reproducible benchmark at $q=4$ propagates every one of the
$n\,2^n$ assignments independently on seeded perturbed convex polygons.  On an
Apple M3, the run grows from 64 assignments in 0.005 seconds for $n=4$ to
10,240 in 5.47 seconds for $n=10$, while reducing the midpoint-anchor counts
from 60 to 10 and from 150 to 60, respectively.  This supports offline use for
modest planar elements, not polynomial-time scalability.  The complete
$n=4,\ldots,10$ table in Online Resource~2 also reports
$S(C)=\sum_{v,\alpha}|C_{v,\alpha}|/L^{2+|\alpha|}$,
$L=\max_v\|v\|_2$.  Since $S(C)$ sometimes increases after support
minimization, fewer terms do not imply better conditioning.

\noindent
{\bf Other planar checks.}
For the non-convex L-polygon, sparse affine anchors give 54 terms and finite search over the
$6\cdot 2^6=384$-assignment gives 30; both reproduce its five exact
quartic moments.  For the seeded non-symmetric convex hexagon, the same class reduces
$90$ to $30$ terms at $q=4$ and reduces $270$ to $90$ at $q=8$.  Online Resource~2
records the vertices, anchors, terminal supports, complete coefficient tables,
scripts, benchmark output, and CSV files.
The identical minimizing anchor pattern in these two six-edge examples was
reproduced by independent runs.  Their common oriented six-cycle incidence
structure makes the selected endpoint anchors prune the same branches, even
though their coefficient values and moment vectors are different.

\section{Conclusion}
We have given a direct face-to-vertex construction of the reusable
coefficient functional which is implicit 
in previous algorithmic implementations of Euler--Stokes HNI.  
While terminal vertex-derivative evaluations already occur in the
established HNI descent, the contribution here is that equal terminal
functionals are explicitly collected and their coefficients retained from the
oriented face geometry before a moment basis or polynomial is selected.  It
provides a reusable functional adapted to vertex-derivative data, exposes
anchor dependence and nonuniqueness, enables exact restricted-information
post-processing and Riesz certification, and yields the degree-$q$/degree-$2q$ coefficient
relation.  In addition, the structured lift 
is particularly useful when
integrands are naturally supplied through local vertex data rather than
global monomial coefficients.  
The four-cube example illustrates  that dimension is not restricted to $2,3$, while the planar
tests show that finite searches over vertex anchors can prune this structured
output offline.  

Future work should identify FEM, VEM, or
Hermite-type settings in which this information model is advantageous and
seek scalable anchor selection with simultaneous sparsity and conditioning
control.\\

\noindent
{\bf Acknowledgement.} 
The author used OpenAI ChatGPT to assist with brainstorming and drafting, as well as generation of 
associated Python algorithmic implementations.  Python scripts and their outputs were checked using exact arithmetic and independent validations described in the manuscript. The author is solely responsible for the final content and conclusions.\\

\noindent
\textbf{Code and data availability.} 
The Python implementation and the input data for the numerical examples are publicly available on 
Zenodo at {\tt https://zenodo.org/}, DOI: 10.5281/zenodo.23086879


\bibliographystyle{spmpsci}
\bibliography{hni_jsc_short}
\end{document}